\documentclass[leqno,11pt]{article}
\usepackage{amssymb,amsmath,amsthm,amsfonts}
\usepackage{verbatim}
\allowdisplaybreaks

\usepackage[utf8]{inputenc}

\usepackage{enumitem}

\usepackage{hyperref}

\theoremstyle{plain}
\newtheorem{theorem}{Theorem}
\newtheorem{lemma}[theorem]{Lemma}
\newtheorem{corollary}[theorem]{Corollary}
\newtheorem{proposition}[theorem]{Proposition}

\theoremstyle{definition}
\newtheorem{definition}[theorem]{Definition}

\newtheorem{example}[theorem]{Example}

\newtheorem{question}[theorem]{Open Question}

\newcommand{\R}{\mathbb R}
\newcommand{\N}{\mathbb N}

\DeclareMathOperator{\diam}{diam}

\DeclareMathOperator{\Hau}{Hau}
\DeclareMathOperator{\Top}{Top}
\DeclareMathOperator{\im}{im}

\numberwithin{theorem}{section} \numberwithin{equation}{section}

\title{Dimension results for mappings of jet space Carnot groups}
\author{Derek Jung\footnote{Supported  by U.S. Department of Education GAANN fellowship P200A150319. \hfill\break {\it Key Words and Phrases:} projection theorems, sub-Riemannian geometry, jet spaces, Hausdorff dimension \hfill\break {\it 2010 Mathematics Subject Classification:} Primary 28A78, 53C17; Secondary 58A20, 22E25, 26A16} \\
University of Illinois at Urbana-Champaign \\ 
}
\date{\today}

\begin{document}

\raggedbottom

\maketitle

\begin{abstract}
We propose analogues of horizontal and vertical projections for model filiform jet space Carnot groups. 
Every pair consisting of the jet of a smooth function on $\mathbb{R}$ and a vertical hyperplane with first coordinate fixed provides a splitting of a model filiform group, which induces mappings of the group. 
We prove Marstrand-type theorems for these mappings and determine the possible Hausdorff dimensions of images of sets under these mappings. 
Analogues of projections for general jet space Carnot groups could be defined similarly.
\end{abstract}


\tableofcontents
 
\setcounter{section}{0}
\setcounter{subsection}{0}


\section{Introduction}\label{sec:intro-proj}

We propose analogues of horizontal and vertical projections for jet space Carnot groups.  
This continues work over the past decade to obtain Marstrand-type results in the setting of sub-Riemannian geometry \cite{BDF:TE, BFP:PAS}.

The effect of projections in Euclidean space on Hausdorff dimension has been studied for decades.
Marstrand began the study in 1954, essentially proving that the projection of a set onto almost every plane has large Hausdorff dimension relative to that of the set \cite{M:SFG}. 
Later research in this area was performed by Kaufman in 1968 \cite{K:OHD}, Mattila in 1975 \cite{M:HDO}, and Peres and Schlag in 2000 \cite{PS:SOP} among many others.
For a greater discussion on this topic, we refer to the reader to Mattila's recent survey  \cite{M:HDP}. 

In 2012, Balogh, Durand-Cartagena, F\"assler, Mattila, and Tyson defined projections in the first Heisenberg group $\mathbb{H}^1$ \cite{BDF:TE}.
$\mathbb{H}^1$ can be identified with $\R^3$ equipped with a group operation and is a simple example of a sub-Riemannian space and of a Carnot group.
For every angle $\theta$, the authors defined $\mathbb{V}_\theta$ to be the horizontal line passing through the origin at angle $\theta$, and $\mathbb{V}^\perp_\theta$ to be its orthogonal complement in $\mathbb{R}^3$.
 The homogeneous subgroups $V_\theta$ and $V^\perp_\theta$  induce a semidirect group splitting of $\mathbb{H}^1$ for each $\theta$, and hence, induce projections $P_{\mathbb{V}_\theta}$, $P_{\mathbb{V}^\perp_\theta}$, respectively.
The main objective of the work in \cite{BDF:TE} was to determine how the Hausdorff dimensions of $P_{\mathbb{V}_\theta}(E)$ and $P_{\mathbb{V}^\perp_\theta}(E)$ compare to the Hausdorff dimension of $E$ for a Borel set $E$, much like in the Euclidean case. 
It should be noted that while each $P_{\mathbb{V}_\theta}$ is linear, a projection, and a group homomorphism, the mappings $P_{\mathbb{V}^\perp_\theta}$ are none of these in general. 
This illustrates the increased difficulty as one studies projection theorems in the sub-Riemannian setting.
The situation will be even worse for the mappings in our present study.

The choice of the vertical planes in \cite{BDF:TE} was motivated by the fact that the restriction of the gauge metric is a snowflaked metric, which makes calculations much easier. 
This will motivate the choice of our vertical planes. 
By the identities proven in Proposition \ref{jet-ind-lemma}, we observe that the gauge metric restricted to  vertical hyperplanes with first coordinate fixed,  is snowflaked as well. 
Hence, we will take these as the images of our vertical projections. 

For the horizontal sets, the Carnot group structure of $J^k(\R)$ supplies us with a rich family of sets to complement these planes, in fact a $C^\infty(\R)$-worth of such sets!
Before we say what they are, we first define some notation. 
$J^k(\R)$ is the space of $k^{th}$-order Taylor polynomials of smooth functions on $\R$; it may be identified with $\R^{k+2}$. 
This can be equipped with a group operation $\odot$ to become a Carnot group. 
For each $f\in C^\infty(\R)$, the path $j^k(f):\R\to J^k(\R)$ is horizontal and locally biLipschitz (Proposition \ref{RW-obs}), where $j^k_x(f)$ is the $k^{th}$-order Taylor polynomial of $f$ at $x$. 

Now we can define our horizontal and vertical mappings. 
For each $t\in \R$, $f\in C^\infty(\R)$, and $p\in J^k(\R)$, there exist unique points  $J_{f,t}(p),$ $V_{f,t}(p)$ in the plane $\{x=t\}:= \{(x,u_k,\ldots ,u_0)\in J^k(\R): x=t\}$, the image of $j^k(f)$, respectively, such that $p = V_{f,t}(p) \odot J_{f,t}(p)$. 
We first show that each $V_{f,t}$ is  biLipschitz when restricted to a hyperplane $\{x=t\}$. 
The map $V_{f,t}$ has a complicated definition involving right multiplication that makes it difficult to analyze from a projection of sets viewpoint.
The mappings $V_{f,t}$ and $J_{f,t}$ are rarely idempotent, hence it would be wrong to call them projections. 
Nevertheless, it is helpful (at least for the author) to think them of as projections.
We show in Section \ref{sec:regularity}  that the mappings share some regularity. 
For all $f\in C^\infty(\R)$ and $t\in\R$,  $V_{f,t}$ is locally $\frac{1}{k+1}$-H\"older (Proposition \ref{Vft-Holder}) while  $J_{f,t}$ is locally biLipschitz (Proposition \ref{loc-Lip-proj}). 

For all $t\in \R$, $f\in C^\infty(\R)$, and Borel sets $E\subset J^k(\R)$,
\[
	0\le \dim_{\Hau} J_{f,t}(E) \le \dim_{\Hau} (\{j_x^k(f): x\in \R\}) = 1
\]
and 
\[
	0 \le \dim_{\Hau} V_{f,t} (E) \le \dim_{\Hau} (\{x=t\}) = \frac{(k+1)(k+2)}{2}. 
\]
The main questions we study in Sections \ref{sec:sets-large-proj} and \ref{sec:sets-null-proj} are: If $t\in \R$ and $f\in C^\infty(\R)$, for which pairs $(\dim_{\Hau} J_{f,t}(E) ,\dim_{\Hau}V_{f,t}) \in [0,1]\times \left[ 0,\frac{(k+1)(k+2)}{2}\right]$ are attained as $E$ varies over Borel sets in $J^k(\R)$? 
Can $E$ be chosen independently of $t$ or $f$?
If $E$ is allowed to depend on $t$ and $f$, we show that all pairs are possible (Theorem \ref{satisfying-theorem}). 
We can choose $E$ independently of $f$ if $\alpha=0$ (Corollary \ref{Vft-cor}). 
Sets independent of $t$ are more difficult to obtain due to the form of $V_{f,t}$. 
However, we show in Example \ref{line-example} that for all nonconstant linear polynomials $f$ and $0\le \alpha \le 1$, there exists a Borel set $E\subset J^k(\R)$ (independent of $t$) satisfying
\[
	\dim_{\Hau}(E) = \alpha, \quad \dim_{\Hau}(J_{f,t}(E)) = \alpha,\quad\text{and} \quad \dim_{\Hau}(V_{f,t}(E)) = (k+1)\alpha
\]
for all $t\in \R$.

We then study the question: For $f\in C^\infty(\R)$ and $t\in \R$, what are the possible pairs 
\[
	(\dim_{\Hau}(E), \dim_{\Hau}(J_{f,t}(E))) \in [0,\dim_{\Hau}J^k(\R)]\times [0,1]
\] 
and 
\[
	(\dim_{\Hau}(E), \dim_{\Hau}(V_{f,t}(E))) \in [0,\dim_{\Hau}J^k(\R)]\times [0,\dim_{\Hau}J^k(\R) - 1]
\]
 for $E$ varies over Borel sets in $J^k(\R)$?
For this question, one needs to take into account that each $J_{f,t}$ is locally Lipschitz and each $V_{f,t}$ is locally $\frac{1}{k+1}$-H\"older. 
If we take this into account, we show that all   pairs $(\dim_{\Hau}(E), \dim_{\Hau}(J_{f,t}(E)))$ are possible, albeit highly dependent on $t$ and $f$ (Proposition \ref{full-range-dim-Jft}).
We only answer the question for the possible pairs $(\dim_{\Hau}(E), \dim_{\Hau}(V_{f,t}(E)))$ when $f$ is a nonconstant linear polynomial (Proposition \ref{linear-k1}). 

In Section \ref{sec:top-dim}, we study the question: For $t\in \R$ and $f\in C^\infty(\R)$, what are the possible pairs $(\dim_{\Top} J_{f,t}(E), \dim_{\Top}V_{f,t}(E))$ as $E$ varies over Borel sets?
Noting that $\dim_{\Top} \text{im}(j^k(f)) =1$ for all $f\in C^\infty(\R)$ and $\dim_{\Top}(\{x=t\})= \frac{(k+1)(k+2)}{2}$, we show that all possible pairs are possible.
In fact, we show that the pair-attaining $E$ can be typically chosen independenty of $f$.
We prove that for all $t\in \R$ and $(a,b) \in \{0,1\}\times \{1,\ldots,k+1\}$, there exists a set $E$ such that $\dim_{\Top}J_{f,t}(E) = a$ and $\dim_{\Top} V_{f,t}(E) = b$ for all $f\in C^\infty(\R)$ (Proposition \ref{top-ind-f}). 
We refer the reader to \cite{HW:DT} for a discussion on topological dimension, which is sometimes  referred to as inductive dimension. 

We conclude by stating a few open questions concerning how $\dim_{\Hau}J_{f,t}(E)$ and $\dim_{\Hau} V_{f,t}(E)$ changes as $f$ or $t$ varies. 
We also remark that one could define these mappings similarly for general jet space Carnot groups $J^k(\R^n)$.

And, of course,  immediately following this will be the (obligatory) background section, where we briefly define the Carnot group structure of $J^k(\R)$, lay out notation, and describe the few crude tools we have to estimate distances in $J^k(\R)$.


\section{Structure of jet space Carnot groups}\label{sec:structure-proj}

We begin this section by defining the model filiform jet space Carnot groups. 
This will follow the explanation by Warhurst in Section 3 of \cite{JSA:W}. 
For a more detailed introduction, we refer the reader to Warhurst's paper  and to Section 2 of \cite{J:BE}.
We will then transition to describing the metric structure of $J^k(\R)$. 

Fix $k\ge 1$.
As a set, $J^k(\R)$ equals the set of $k^{th}$-order Taylor polynomials of smooth functions on $\R$.
As a $k^{th}$-order Taylor polynomial depends entirely on a base point and its derivatives up to order $k$ at the point(including the $0^{th}$ derivative which is just the value of the function), $J^k(\R)$ may be identified with $\R^{k+2}$.
In Section 3 of \cite{JSA:W}, Warhurst describes how one can naturally define a Carnot group structure on $J^k(\R)$.
We will outline his construction here. 

$J^k(\R)$ is $\R^{k+2}$ equipped with the operation
\[
	(x,u_k,\ldots , u_0)\odot (y,v_k,\ldots , v_0) = (z,w_k,\cdots ,w_0),
\]
where $z=x+y,$ $w_k = u_k+v_k$, and 
\[
	w_s = u_s + v_s + \sum_{j=s+1}^k u_j\frac{y^{j-s}}{(j-s)!}, \qquad s=0,\ldots , k-1.
\]
$J^k(\R)$ forms a Carnot group when equipped with this group operation. 
The horizontal bundle $HJ^k(\R)$ is defined by the contact $1$-forms 
\[
	\omega_j := du_j -u_{j+1}dx , \qquad j=0,\ldots , k-1,
\]
and is globally framed by the left-invariant vector fields $X^{(k)} := \frac{\partial}{\partial x} + u_k \frac{\partial}{\partial u_{k-1}} + \cdots + u_1\frac{\partial}{\partial u_0}$ and $\frac{\partial }{\partial u_k}$. 
The Lie algebra of $J^k(\R)$ admits a $(k+1)$-step stratification:
\[
	Lie(J^k(\R))  = \left\langle X^{(k)} , \frac{\partial}{\partial u_k} \right\rangle \oplus \left\langle \frac{\partial }{\partial u_{k-1}} \right\rangle\oplus \cdots \oplus \left\langle \frac{\partial}{\partial u_0}\right\rangle.
\]
The minimal  dimension of $Lie(J^k(\R))$ relative to its step is the reason behind why these groups are referred to as model filiform. 

By a theorem of Chow \cite{C:US}, $J^k(\R)$ is horizontally path-connected.
More specifically, for every $p,q\in J^k(\R)$, there exists a path $\gamma$ connecting $p$ to $q$ that is absolutely continuous as a map into $\R^{k+2}$ with derivative horizontal a.e. 
Thus, we can define the \textbf{Carnot-Carath\'eodory metric} $d_{cc}$ on $J^k(\R)$, where $d_{cc}(p,q)$ is defined to be the infimum of lengths of horizontal curves connecting $p$ to $q$ in $J^k(\R)$. 
It is well-known that this forms a geodesic, left-invariant metric on $J^k(\R)$ that is one-homogeneous with respect to the dilations on $J^k(\R)$ given by
\[
	\delta_\epsilon(x,u_k,u_{k-1}, \ldots , u_0) = (\epsilon x, \epsilon u_k ,\epsilon^2u_{k-1}, \ldots , \epsilon^{k+1} u_0), \qquad \epsilon>0. 
\]
We will often refer to the Carnot-Carath\'eodory metric as the CC-distance for brevity.

In the Heisenberg groups $\mathbb{H}^n$, geodesics are very well-understood. 
For any point $p$ away from the vertical axis, there exists a unique geodesic connecting the origin to $p$.
If $p$ is on the vertical axis, there is a family of geodesics connecting the origin to $p$.
In fact, exact formulas are known for these geodesics and can be used to show that the function on $\mathbb{H}^n$ measuring distance from the origin, is analytic away from the vertical axis \cite[Corollary 2.7 and Theorem 3.1]{HZ:G}. 
This makes it much easier to estimate distances in the Heisenberg groups.

The situation is far worse in $J^k(\R)$. 
Exact formulas or even the regularity of geodesics is not well-understood in the model filiform groups.
In fact, Le Donne, Pinamonti, and Speight proved that the CC-distance is not even Pansu differentiable in horizontal directions  in the Engel group (which is isomorphic to $J^2(\R)$) \cite[Theorem 4.2]{LDPS:UDS}. 
This lack of understanding stems from the presence of abnormal geodesics in $J^k(\R)$ when $k\ge 2$. 
This makes it much harder to estimate distances in the model filiform groups than in the Heisenberg groups. 
We will need cruder tools to estimate distances in jet spaces.

First, the result of Nagel, Stein, and Wainger at \cite[Proposition 1.1]{NSW:BAM} for general Carnot groups implies that the identity map $\text{id}:J^k(\R)\to \R^{k+2}$ is locally $\frac{1}{k+1}$-H\"older with locally Lipschitz inverse. 
From this, one can use the homogeneity of the CC-distance to obtain the Ball-Box Theorem, which states here that the CC-balls $B_{cc}(r) := \{p\in J^k(\R): d_{cc}(p) \le r\}$ are uniformly equivalent to the boxes $[-\epsilon,\epsilon]^2 \times \prod_{j=2}^{k+1} [-\epsilon^j,\epsilon^j].$
Stated otherwise, one is able to estimate the distance of a point from the origin by  the point's coordinates.
\begin{corollary}\label{cor-bb}
Fix $k\ge 1$.
There exists a constant $C>0$ such that for all $(x,u_k,\ldots , u_0)\in J^k(\R)$,
\[
	\frac{1}{C} \cdot d_{cc}(0,(x,u_k,\ldots , u_0)) \le \max\{|x|, |u_k|, |u_{k-1}|^{1/2} , \ldots , |u_0|^{1/(k+1)}\} \le C\cdot d_{cc}(0,(x,u_k,\ldots ,u_0)). 
\]
\end{corollary}

By left-invariance, $d_{cc}(p,q) = d_{cc}(0,p^{-1}\odot q)$.
 We observe that one could utilize Corollary \ref{cor-bb} to estimate the CC-distance between two points $p$ and $q$ if one knew the form of $p^{-1}\odot q$.
That's the point of the following proposition. 
We write $p_x$ below to denote the first coordinate of a point $p\in J^k(\R)$. 
We will use this notation throughout this paper and it should (hopefully) not cause any confusion.

\begin{proposition}\label{jet-ind-lemma}
For all $(x,u_k,\ldots , u_0) , $ $ (y,v_k,\ldots , v_0)\in J^k(\R)$,
\[
	((x,u_k,\ldots , u_0) ^{-1}\odot(y,v_k,\ldots , v_0))_x = y-x
\]
and
\[
((x,u_k,\ldots , u_0) ^{-1}\odot(y,v_k,\ldots , v_0))_s=v_s - u_s - \sum_{j=s+1}^k \frac{(y-x)^{j-s}}{(j-s)!}\cdot u_j,\qquad s= 0, 1, \ldots , k. 
\]
\end{proposition}

\begin{proof}
It is easy to check
\[
((x,u_k, \ldots , u_0)^{-1})_t = - \sum_{j=t}^k \frac{ (-x)^{j-t} }{(j-t)!} u_j, \quad t= 0,\ldots , k.
\]
Thus, 
\begin{align*}
((x,u_k,\ldots , u_0)^{-1}\odot ( y,v_k,\ldots , v_0))_s &= v_s - \sum_{t=s}^k \sum_{j=t}^k \frac{y^{t-s}}{(t-s)!} \cdot 
\frac{(-x)^{j-t}}{(j-t)!} \cdot u_j\\
& = v_s - \sum_{j=s}^k \sum_{t=s}^j \frac{y^{t-s}}{(t-s)!} \cdot 
\frac{(-x)^{j-t}}{(j-t)!} \cdot u_j\\ 
& = v_s - \sum_{j=s}^k \sum_{t=s}^j \binom{j-s}{t-s}	y^{t-s}(-x)^{j-t}	\cdot \frac{u_j}{(j-s)!}	\\
& = v_s - \sum_{j=s}^k \frac{(y-x)^{j-s}}{(j-s)!} \cdot u_j, 
\end{align*}
where the last equality comes from the Binomial Theorem.
\end{proof}

Our other main tool for estimating distances will be an observation from Rigot and Wenger. 
First, we need to establish some notation.
\begin{definition}{(jets of functions)}
For all $f\in C^{k+1}(\R)$, define the path  $j^k(f):\R\to J^k(\R)$ by
\[
	j^k(f)(x) := (x,f^{(k)}(x), f^{(k-1)}(x), \ldots , f(x)). 
\]
We will usually write $j^k(f)(x)$ as $j_x^k(f)$.
\end{definition} 

Due to  the construction of $J^k(\R)$, the jet of a $C^{k+1}$-smooth function is a horizontal curve.
By calculating the length of a subcurve, one can then directly estimate distances in  jet spaces.
This is what Rigot and Wenger observed.
\begin{proposition}\label{RW-obs}\cite[Pages 4-5]{RW:LNE} 
Fix $f\in C^{k+1}(\R)$.
The path $j^k(f):\R\to J^k(\R)$ is horizontal.
Moreover, there exists $C>0$ such that for all $x,y \in \R$,
\[
	\frac{1}{C} \cdot |x-y| \le d_{cc}(j_x^k(f),j_y^k(f)) \le \sup_{t\in [x,y]} \left( 1 + (f^{(k+1)} (t))^2 \right)^{1/2}|x-y|.
\]
In particular, $j^k(f):\R\to J^k(\R)$ is locally biLipschitz.
\end{proposition}

\begin{proof}
For $f\in C^{k+1} (\R)$, $j^k(f)$ is $C^1$ and horizontal  with
\[
	(j^k(f))'(x) = X^{(k)}_{j_x^k(f)} +\left. f^{(k+1)}(x)\cdot  \frac{\partial}{\partial u_k}\right|_{j_x^k(f)}.
\]
By the definition of the CC-distance, 
\[
	d_{cc}(j_x^k(f), j_y^k(f)) \le \text{length} (j^k(f)|_{[x,y]})  \le \sup_{t\in [x,y]} \left( 1 + (f^{(k+1)} (t))^2 \right)^{1/2}|x-y|
\]
for all $x,y\in \R$. 
The other inequality follows from Corollary \ref{cor-bb} and the fact
\[
	((x,u_k,\ldots , u_0) ^{-1}\odot(y,v_k,\ldots , v_0))_x = y-x.
\] 
The final statement follows from the fact that $f^{(k+1)}$ is bounded on compact sets. 
\end{proof}


\section{Horizontal and vertical mappings}\label{sec:havm}

Before we define our mappings, we will first step back and consider how Balogh, Durand-Cartagena, F\"assler, Mattila, and Tyson  defined horizontal and vertical projections in the first Heisenberg group in \cite{BDF:TE}.  
This is, after all, the motivation behind how we will define our horizontal and vertical mappings.

The first Heisenberg group $\mathbb{H}^1$ is $\R^{3}$ equipped with the group law
\[
	(z,t) \ast (z',t') = (z+z' , t+t' + 2\omega(z,z')),
\]
where $\omega$ is  the standard symplectic form defined by  $\omega((x,y), (x',y')) :=  yx' -y'x.$
The authors of \cite{BDF:TE, BFP:PAS} equip $\mathbb{H}^1$ with the \textbf{Kor\'anyi metric}:
\[
	d_K (p,p') := || p^{-1} \ast p'||, \quad \text{where } ||p||_K := (|z|^4 + t^2)^{1/4}.
\]
It is easy to see that this is equivalent to the gauge metric on $\mathbb{H}^1$:
\[
	d_G(p,p') := ||p^{-1}\ast p'|| , \quad\text{where } ||p|| := |z| + t^{1/2}.
\]
We also note that by the result of Nagel, Stein, and Wainger \cite[Proposition 1.1]{NSW:BAM}, the Kor\'anyi metric is equivalent to the Carnot-Carath\'eodory metric on $\mathbb{H}^n$. 

The authors of \cite{BDF:TE} considered  projections onto horizontal and vertical subgroups in $\mathbb{H}^1$.
Let $V_\theta$ be the (horizontal) line in $\R^2\times \{0\}$ passing through the origin at angle $\theta$, and
note that the Kor\'anyi metric agrees with the standard Euclidean metric on $V_\theta$.
Moreover, the Kor\'anyi metric takes on a simple form on the orthogonal complement $V_\theta^\perp$:
\[
	d_K ((aie^{i\theta},t) , (a'ie^{i\theta},t')) = ( |a-a'|^4 + (t-t')^2)^{1/4}.
\]
Finally, for each $\theta$, we have a semidirect group splitting  $\mathbb{H}^1= V_\theta \rtimes V_\theta^\perp$. 
In particular, for all $p \in \mathbb{H}^1$, there exist unique $P_{V_\theta}(p) \in V_\theta$ and $P_{V_\theta^\perp}(p) \in V_\theta^\perp$ such that
\[	
	p = P_{V_\theta}(p) \ast P_{V_\theta^\perp}(p). 
\]
They then proceed to consider the effect of $P_{V_\theta}$ and $P_{V_\theta^\perp}$ on Hausdorff dimension.

We will seek to find a similar splitting of our jet spaces. 
Rather than a Kor\'anyi metric, we will use a gauge distance on $J^k(\R)$.
\begin{definition}{(Gauge distance of $J^k(\R)$)}
For all $p,q \in J^k(\R)$, define the \textbf{gauge distance}
\[
	d(p,q) := ||p^{-1}\odot q||, \qquad\text{where } ||(x,u_k,\ldots ,u_0)||:= |x| + \sum_{j=0}^k|u_j |^{1/(k+1-j)}.
\]
\end{definition}
While the gauge distance isn't an actual metric in the metric space sense, it is  equivalent to the CC-metric thanks to Corollary \ref{cor-bb}. 
\begin{proposition}\label{Gauge-cc}
There exists a constant $C>0$ such that for all $p,q\in J^k(\R)$, 
\[
	\frac{ d(p,q)}{C} \le d_{cc}(p,q) \le Cd(p,q).
\]
\end{proposition}

In hoping to replicate the construction for $\mathbb{H}^1$, we will first find a family of vertical sets on which the restriction of the gauge distance takes on a simple form.  
Suppose $(x,u_k,\ldots , u_0), (y,v_k,\ldots  , v_0)$ are two points such that 
\begin{equation}\label{simple-form}
	((x,u_k,\ldots ,u_0)^{-1}\odot (y,v_k,\ldots , v_0))_s = v_s -u_s \qquad \text{for all } s=0,\ldots , k.
\end{equation}
Then
\[
	d((x,u_k,\ldots ,u_0)^{-1}\odot (y,v_k,\ldots , v_0)) = |x-y| + \sum_{j=0}^k |v_s-u_s|^{1/(k+1-j)}. 
\]
(We think that this would be a pretty simple form for the distance!)
From Proposition \ref{jet-ind-lemma}, we see that (\ref{simple-form}) holds if $x=y$. 
Thus, for our vertical sets, we will choose the planes 
\[
	\{x=t\} := \{(x,u_k,\ldots , u_0)\in J^k(\R): x=t\}, \qquad t\in \R.
\]

We now seek a horizontal set that induces a splitting of $J^k(\R)$ when coupled with each of the planes $\{x=t\}$.
Fortunately, by the Carnot group structure of $J^k(\R)$, we have a whole $C^{k+1}(\R)$-family of such sets- images of jets of functions in $C^{k+1}(\R)$!
For simplicity, we will primarily consider functions in $C^\infty(\R)$ in this paper, but it would be interesting to explore if anything would change if we allowed all functions in $C^{k+1}(\R)$. 

Fix $f\in C^\infty(\R)$.
For all $p= (x,u_k,\ldots , u_0) \in J^k(\R)$,
\begin{equation}\label{splitting-product}
	p = (p\odot j_{x-t}^k(f)^{-1}) \odot j_{x-t}^k(f),
\end{equation}
where $p\odot j_{x-t}^k(f)^{-1}  \in \{x=t\}$ and $j_{x-t}^k(f) \in \text{im}(j^k(f)) := \{ j_y^k(f):y\in \R\}$.
Moreover, it isn't hard to see that $ p\odot j_{x-t}^k(f)^{-1} , j_{x-t}^k(f)$ are the unique points in $\{x=t\}, \text{im}(j^k(f))$, respectively, for which (\ref{splitting-product}) holds. 
Indeed, suppose
\[
	(x,u_k,\ldots , u_0) = q \odot j_s^k(f)
\]
for some $q \in \{x=t\}$ and $s\in \R$. 
Then
\[
	s = (j_s^k(f))_x = (q^{-1}\odot (x,u_k,\ldots , u_0))_x = x-t .
\]
From there,
\[
	q = (x,u_k,\ldots , u_0) \odot j_{x-t}^k(f)^{-1}. 
\]
For each $f\in C^\infty(\R)$ and $t\in \R$, we define the \textbf{vertical mapping} $ V_{f,t}:J^k(\R)\to \{x=t\}$  by 
\[
	V_{f,t} (x,u_k,\ldots , u_0 ) := (x,u_k,\ldots , u_0)\odot j_{x-t}^k(f)^{-1}
\]
and the \textbf{horizontal mapping} $J_{f,t}:J^k(\R)\to \im (j^k(f))$  by 
\[
	J_{f,t} (x,u_k,\ldots , u_0) := j_{x-t}^k(f). 
\]
Then
\[
	p = V_{f,t} (p) \odot J_{f,t}(p)
\]
for all $p\in J^k(\R)$.

The planes $\{x=t\}$ are clearly not vertical subspaces or closed under $\odot$ unless $t=0$.
Also, the horizontal sets $\text{im}(j^k(f))$ are not subgroups of $J^k(\R)$ in general. 
However, by left-invariance, the left-cosets $p\odot \text{im}(j^k(f))$ are isometrically equivalent with respect to the CC-distance. 

We emphasize that these mappings are not linear projections much less idempotent in general.
 More specifically, it is not the case that $V_{f,t} = V_{f,t} \circ V_{f,t}$ or $J_{f,t} = J_{f,t} \circ J_{f,t}$ in general.
It is true that
\[
	\text{im}(V_{f,t}) = \{x=t\} \qquad \text{and} \qquad \text{im}(J_{f,t}) = \text{im}(j^k(f)).
\]
However, for all $p\in \{x=t\}$,
\[
	V_{f,t}(p) = p\odot j_0^k(f)^{-1} = p
\]
if and only if $j_0^k(f) = 0$.
Also, for all $j_x^k(f) \in \text{im}(j^k(f))$,
\[
	J_{f,t}(j_x^k(f)) = j_{x-t}^k(f) = j_x^k(f)
\]
if and only if $t=0$. 
It follows that   $V_{f,t} = V_{f,t} \circ V_{f,t}$ if and only if $j_0^k(f) =0$, and $J_{f,t} = J_{f,t}\circ J_{f,t}$ if and only if $t=0$. 

As in the Heisenberg group, we will be interested in the effect of $V_{f,t} $ and $J_{f,t}$ on dimensions of Borel sets, both topological and Hausdorff.
We will also be interested in the possibilities of pairs $(\dim_{\Hau}V_{f,t}(E),\dim_{\Hau}J_{f,t}(E))$ as $E$ ranges over Borel sets. 
The beauty of studying the effect of these mappings on sets is that we now have three parameters to play with:
the Borel set $E$, the hyperplane parameter $t$, and the smooth function $f$.


\subsection{Simple examples}\label{sec:simple-proj}

Before we dive too deeply into studying the maps $J_{f,t}$ and $V_{f,t}$, we will consider a couple of examples that, at first thought, should be simple in terms of studying the Hausdorff dimensions of their images under $J_{f,t} $ and $ V_{f,t}$.

\begin{example}\label{plane-example}
Fix $t\in \R$. 
For $f\in C^\infty(\R)$, we will consider  the horizontal image $J_{f,t}(\{x=t\})$ and the vertical image $V_{f,t}(\{x=t\})$ of the plane $\{x=t\}$.
In this example, we will obtain dimension results that are independent of $f$. 

For all $p\in \{x=t\}$, $J_{f,t}(p) = j_0^k(f).$
Hence, 
\[
	J_{f,t}(\{x=t\}) = \{j_0^k(f)\},
\] 
and
\[
	\dim_{\Hau} (J_{f,t}(\{x=t\}) = \dim_{\Top} (J_{f,t}(\{x=t\})) = 0.
\]

By Lemma \ref{right-multiplication}, for all $p\in \{x=t\}$, 
\[
	V_{f,t} (p) = p \odot j_0^k(f)^{-1} = p - j_0^k(f).
\]
Here, $p-j_0^k(f)$ represents the vector difference of $p$ and $j_0^k(f)$ when both are viewed as elements of $\R^{k+2}$. 
This implies
\[
	V_{f,t}(\{x=t\}) = \{x=t\} - j_0^k(f)= \{x=t\}.
\]
We may conclude
\[
	\dim_{\Top} (V_{f,t}(\{x=t\}))  = \dim_{\Top}(\{x=t\}) = k+1	
\]
and 
\[
		\dim_{\Hau} (V_{f,t}(\{x=t\})) = \dim_{\Hau} (\{x=t\})  = \frac{(k+1)(k+2)}{2},
\]
where both equalities are independent of the function $f$.
\end{example}

A careful examination of our work shows something remarkable: the restriction of $V_{f,t}$ to the plane $\{x=t\}$  is given by subtraction by a fixed vector. This is a homeomorphism of $\{x=t\}$, which implies that $V_{f,t}$ preserves the topological dimension of subsets of $\{x=t\}$. 
Moreover, the gauge distance behaves very well with respect to subtraction by a fixed element: 
\[
	d(V_{f,t}(p) , V_{f,t}(q)) = d(p - j_0^k(f) , q-j_0^k(f)) = d(p,q), \quad p,q\in \{x=t\}.
\]
By Proposition \ref{Gauge-cc}, $V_{f,t}|_{\{x=t\}}$ being  $d$-isometric implies that $V_{f,t}|_{\{x=t\}} $ is  $d_{cc}$-biLipschitz. 
We  mark all of this down  in a proposition.

\begin{proposition}\label{bilip-Vft}
Fix $t\in \R$ and $f\in C^\infty(\R)$.
Then 
\[
	\dim_{\Top}(V_{f,t}(E)) = \dim_{\Top}(E) 
\]
for all $E\subset \{x=t\}$.
Moreover, the restriction $V_{f,t}|_{\{x=t\}}:\{x=t\}\to \{x=t\}$ is biLipschitz when $\{x=t\}$ is equipped with the  restriction of the Carnot-Carath\'eodory distance. 
\end{proposition}

We make the remark  that it won't make a difference whether we use $d$ or $d_{cc}$ to compute Hausdorff dimension as they are biLipschitz equivalent by Proposition \ref{Gauge-cc}.
It will be useful for the reader to keep this in mind  throughout this paper (and possibly in life in general).

We conclude this section with the complementary example of a set: the image of a jet. 

\begin{example}\label{jet-example-sec3}
Fix $f\in C^\infty(\R) $ and $t\in \R$. 
In this example, we will consider applying our mappings to  the image $\text{im}(j^k(f))$ of the jet $j^k(f)$. 
At first glance, it might seem like things will be similar to the previous example and $V_{f,t}(\text{im}(j^k(f)))$ and $J_{f,t}(\text{im}(j^k(f)))$ will be simple to study from a dimension standpoint. 
And, in fact, $J_{f,t}(\text{im}(j^k(f)))$ is pretty easy to study.
For all $x\in \R$,
\[
	J_{f,t}(j_x^k(f)) = j_{x-t}^k(f) ,
\]
so that
\[
	J_{f,t}(\text{im}(j^k(f))) = \text{im}(j^k(f)). 
\] 

However, the study of $V_{f,t}(\text{im}(j^k(f)))$ is a bit more complicated than one would first expect.
One has 
\[
	V_{f,t}(j_x^k(f)) = j_x^k(f) \odot j_{x-t}^k(f)^{-1}, \qquad x\in \R.
\]
By Lemma \ref{right-multiplication} (to be proven in the next section), 
\[
	 (j_x^k(f) \odot j_{x-t}^k(f)^{-1} )_s = \sum_{j=s}^k \frac{(t-x)^{j-s}}{(j-s)!} (f^{(j)}(x) - f^{(j)}(x-t)) , \qquad s=0,\ldots , k.
\]
Unless $f$ is a constant function, $V_{f,t}(\text{im}(j_x^k(f)))$ will be a smooth, nonconstant curve, hence have topological dimension $1$. 
However, it isn't clear what its Hausdorff dimension is. 
\end{example}


\subsection{Regularity of $J_{f,t}$ and $V_{f,t}$}\label{sec:regularity}

Now that we have seen a couple examples and played around a little with the maps, we will prove our first result. 
We will show that the horizontal mappings and vertical mappings share some regularity amongst themselves.
As one might expect,  each of the mappings $J_{f,t} $ is locally Lipschitz and each of the $V_{f,t}$ is locally $\frac{1}{k+1}$-H\"older. 

The proof for $J_{f,t}$ is much simpler, so we will begin there.
And in fact, it should be expected that the proof will be easier for $J_{f,t}$ since each maps to a $1$-dimensional subset of $J^k(\R)$ and is given by essentially shifting the $x$-coordinate of a point by $t$. 

\begin{proposition}\label{loc-Lip-proj}
For all $t\in \R$ and $f\in C^\infty(\R)$, $J_{f,t}:J^k(\R) \to \im(j^k(f))$ is locally Lipschitz. 
\end{proposition}

\begin{proof}
Fix $t\in \R$ and $f\in C^\infty(\R)$, and let  $(x,u_k,\ldots , u_0), (y,v_k,\ldots , v_0) \in J^k(\R)$ be given. 
By Proposition \ref{RW-obs} and Corollary \ref{cor-bb},
\begin{align*}
	d_{cc}(J_{f,t}&(x,u_k,\ldots , u_0), J_{f,t}(y,v_k,\ldots , v_0)) \\
& = d_{cc} (j_{x-t}^k(f), j_{y-t}^k(f))\\
& \le \sup_{s\in [x-t,y-t]} \left( 1 + (f^{(k+1)}(s))^2\right)^{1/2} |x-y|\\
&\le C\sup_{s\in [x-t,y-t]} \left( 1 + (f^{(k+1)}(s))^2\right)^{1/2}d_{cc}((x,u_k,\ldots , u_0), (y,v_k,\ldots , v_0)),
\end{align*}
where $C$ is the constant from Corollary \ref{cor-bb}. 
As $f^{(k+1)}$ is bounded on compact sets,  $J_{f,t}$ is locally Lipschitz. 
\end{proof}

As an immediate corollary, we obtain
\begin{corollary}\label{Jft-cor}
For all Borel sets $E\subset J^k(\R)$, $t\in \R$, and $f\in C^\infty(\R)$, 
\[
	\dim_{\Hau} (J_{f,t}(E)) \le \min\{\dim_{\Hau}(E),1\}. 
\]
\end{corollary}

We see that the proof for $J_{f,t}$ being locally Lipschitz follows pretty easily from  Proposition \ref{RW-obs} and Corollary \ref{cor-bb}.
However, things get a bit more difficult when we shift to analyzing $V_{f,t}$.
$V_{f,t}$ maps to a hyperplane as opposed to a curve, and also $V_{f,t}$ involves a right-translation, which is notorious for being unwieldy.
Fortunately, the simple form of the group operation on $J^k(\R)$ will save us.
We first prove the form of $p\odot q^{-1}$ for $p,q\in J^k(\R)$, with the motivation of doing so being the particular form of $V_{f,t}$. 

\begin{lemma}\label{right-multiplication}
For all $(x,u_k,\ldots , u_0), (y,v_k,\ldots , v_0) \in J^k(\R)$, 
\[
	((x,u_k,\ldots , u_0) \odot (y,v_k,\ldots v_0)^{-1})_s = \sum_{j=s}^k \frac{(-y)^{j-s}}{(j-s)!} (u_j-v_j) ,\qquad s= 0,\ldots , k.
\]
\end{lemma}

\begin{proof}
First,
\[
	((y,v_k,\ldots , v_0)^{-1})_s= -\sum_{j=s}^k \frac{(-y)^{j-s}}{(j-s)!} v_j, \qquad s=0,\ldots , k.
\]
We can calculate
\begin{align*}
((x,u_k,\ldots , u_0) \odot (y,v_k,\ldots v_0)^{-1})_s&= u_j - \sum_{j=s}^k \frac{(-y)^{j-s}}{(j-s)!} v_j + \sum_{j=s+1}^k  \frac{(-y)^{j-s}}{(j-s)!} u_j\\
&= \sum_{j=s}^k \frac{(-y)^{j-s}}{(j-s)!} (u_j-v_j). 
\end{align*}
\end{proof}

We can now prove that the $V_{f,t}$ are locally H\"older.

\begin{proposition}\label{Vft-Holder}
For all $t\in \R$ and $f\in C^\infty(\R)$, $V_{f,t}:J^k(\R)\to \{x=t\}$ is locally $\frac{1}{k+1}$-H\"older. 
\end{proposition}

\begin{proof}
Let $t\in \R$ and $f\in C^\infty(\R)$ be given. 
It suffices to prove that for all $M>1$,  there exists a constant $C_M$ such that 
\[
	d_{cc}(V_{f,t}(x,u_k,\ldots , u_0) , V_{f,t}(y,v_k,\ldots , v_0)) \le C_Md_{cc}((x,u_k,\ldots ,u_0), (y,v_k,\ldots ,v_0))^{1/(k+1)}
\]
for all $(x,u_k,\ldots , u_0) , (y,v_k,\ldots , v_0) \in [-M,M]^{k+2} \subset J^k(\R)$. 

Fix $M>1$ and $(x,u_k,\ldots , u_0) , (y,v_k,\ldots  ,v_0) \in [-M,M]^{k+2}. $ 
We have
\begin{equation}
\begin{aligned}
	d_{cc}(V_{f,t}&(x,u_k,\ldots , u_0) , V_{f,t}(y,v_k,\ldots , v_0)) \label{Vft-bound} \\
& = d_{cc}(0, V_{f,t}(x,u_k,\ldots , u_0)^{-1} \odot V_{f,t}(y,v_k,\ldots ,v_0) )\\
& \le C\biggl(  | (V_{f,t}(x,u_k,\ldots  , u_0))_x- (V_{f,t}(y,v_k,\ldots , v_0))_x| \\
&\qquad\qquad+ \sum_{s=0}^k | (V_{f,t}(x,u_k,\ldots  , u_0)^{-1}\odot V_{f,t}(y,v_k,\ldots , v_0))_s| ^{1/(k+1-s)}
\biggr) ,
\end{aligned}
\end{equation}
where $C$ is the constant from Proposition \ref{Gauge-cc}. 
Thus, it suffices to bound the coordinates of $V_{f,t}(x,u_k,\ldots  , u_0)^{-1}\odot V_{f,t}(y,v_k,\ldots , v_0)$. 

By definition of $V_{f,t}$, 
\[
	(V_{f,t} (x,u_k,\ldots , u_0) )_x - (V_{f,t}(y,v_k,\ldots , v_0))_x = t - t=0.
\]

By Lemma \ref{right-multiplication},
\[
	  ((x,u_k,\ldots , u_0) \odot j_{x-t}^k(f)^{-1} )_s = \sum_{j=s}^k \frac{(t-x)^{j-s}}{(j-s)!} (u_j -f^{(j)}(x-t))
\]
and
\[
	((y,v_k,\ldots , v_0) \odot j_{y-t}^k(f)^{-1})_s = \sum_{j=s}^k \frac{(t-y)^{j-s}}{(j-s)!} (v_j -f^{(j)}(y-t))
\]
for $s=0 ,\ldots , k$. 
By Proposition \ref{jet-ind-lemma}, this implies
\begin{align*}
	\biggl(((x,u_k,\ldots , u_0) &\odot j_{x-t}^k(f)^{-1} )^{-1}\odot ((y,v_k,\ldots , v_0) \odot j_{x-t}^k(f)^{-1})\biggr)_s\\
& = \sum_{j=s}^k \frac{(t-y)^{j-s}}{(j-s)!} \left( v_j - f^{(j)} (y-t)\right) - \sum_{j=s}^k \frac{(t-x)^{j-s}}{(j-s)!} \left( u_j -f^{(j)}(x-t)\right) .
\end{align*}
(Note that the $x$-coordinates of $V_{f,t}(x,u_k,\ldots , u_0)$ and $V_{f,t}(y,v_k,\ldots,u_0)$ agree, hence most terms drop out.)
Adding and subtracting terms, the last expression can be rewritten as
\begin{equation}
\begin{aligned}
	&\sum_{j=s}^k \frac{(t-y)^{j-s}}{(j-s)!} \left( v_j - f^{(j)} (y-t)\right) -\sum_{j=s}^k \frac{(t-x)^{j-s}}{(j-s)!} \left( v_j - f^{(j)} (y-t)\right) \label{large-sum}\\
&\qquad +\sum_{j=s}^k \frac{(t-x)^{j-s}}{(j-s)!} \left( v_j - f^{(j)} (y-t)\right) - \sum_{j=s}^k \frac{(t-x)^{j-s}}{(j-s)!} \left( u_j -f^{(j)}(x-t)\right) .
\end{aligned}
\end{equation}
By the Binomial Theorem, 
\[
	|a^n -b^n |\le  n(M+|t|)^{n-1}|a-b|
\]
for $-M-t\le a,b\le M-t, \ n\in \N$, 
which implies
\[
	|(t-y)^{j-s}- (t-x)^{j-s}| \le (j-s) (M+|t|)^{j-s-1} |x-y|, \qquad j=s+1,\ldots , k.
\]
If we define 
\[
	A:= \max \{|f^{(j)}(a)|: a\in [-M-t,M-t],\  j= s+1,\ldots, k+1\},
\] 
then
\begin{equation}
\begin{aligned}
	&\sum_{j=s}^k \frac{(t-y)^{j-s}}{(j-s)!} \left( v_j - f^{(j)} (y-t)\right) -\sum_{j=s}^k \frac{(t-x)^{j-s}}{(j-s)!} \left( v_j - f^{(j)} (y-t)\right)\label{first-summand}\\
&\qquad \le \sum_{j=s+1}^k \frac{(j-s) (M+|t|)^{j-s-1} |x-y|}{(j-s)!}\cdot (M+ A)\\
& \qquad  \le (k-s)^2 (M+|t|)^{k-s-1}(M+A) \cdot |x-y|\\
&\qquad \le k^2 (M+|t|)^{k-1}(M+A)  Dd_{cc}((x,u_k,\ldots , u_0), (y,v_k,\ldots , v_0)),
\end{aligned}
\end{equation}
where $D$ is the constant from Corollary \ref{cor-bb}.
This bounds the first expression of (\ref{large-sum}). 

For the second expression, by the Mean Value Theorem,
\begin{equation}\label{MVT-bound}
	|f^{(j)}(y-t)-f^{(j)}(x-t)| \le A |x-y|\le AD d_{cc}((x,u_k,\ldots , u_0), (y,v_k,\ldots , v_0))
\end{equation}
for $j=s,\ldots , k$. 
Moreover, by the theorem of Nagel-Stein-Wainger \cite[Proposition 1.1]{NSW:BAM}, the identity map $\text{id}:J^k(\R)\to \R^{k+2}$ is locally Lipschitz. 
In particular, there exists a constant $D_M$ such that for all $p,q\in [-M,M]^{k+2}\subset J^k(\R)$,
\[
	|\text{id}(p) - \text{id}(q)| \le D_M d_{cc} (p,q).
\]
This implies 
\begin{align*}
	|v_j -u_j|& \le |\text{id}(x,u_k,\ldots ,u_0) - \text{id}(y,v_k,\ldots , v_0)| \\
&\le D_M d_{cc} ((x,u_k,\ldots , u_0) , (y,v_k,\ldots , v_0)). 
\end{align*}
We can combine this with (\ref{MVT-bound}) to obtain
\[
	|(v_j -f^{(j)}(y-t)) - (u_j -f^{(j)} (x-t))| \le (AD+D_M) d_{cc} ((x,u_k,\ldots , u_0) , (y,v_k,\ldots , v_0)). 
\]
By (\ref{large-sum}) and (\ref{first-summand}), 
\begin{align*}
\biggl(((x,u_k,\ldots , u_0) &\odot j_{x-t}^k(f)^{-1} )^{-1}\odot ((y,v_k,\ldots , v_0) \odot j_{x-t}^k(f)^{-1})\biggr)_s \\
& \le k^2 (M+|t|)^{k-1}(M+A)D d_{cc} ((x,u_k,\ldots , u_0) , (y,v_k,\ldots , v_0))\\
&\qquad + \sum_{j=s}^k \frac{(|t|+M)^{j-s}}{(j-s)!}(AD+D_M) d_{cc} ((x,u_k,\ldots , u_0) , (y,v_k,\ldots , v_0))\\
& \le \widetilde{C_M }d_{cc} ((x,u_k,\ldots , u_0) , (y,v_k,\ldots , v_0)),
\end{align*} 
where
\[
	\widetilde{C_M} : = k^2 (M+|t|)^{k-1}(M+A)D+ (k+1)(|t|+M)^{k} (AD+D_M)+1 . 
\]
(We included the extra term of $1$ at the end just to be secure later when we consider roots of $\widetilde{C_M}$.)
By (\ref{Vft-bound}), we have
\begin{align*}
	d_{cc}(V_{f,t}&(x,u_k,\ldots , u_0) , V_{f,t}(y,v_k,\ldots , v_0))\\
&\le C\sum_{s=0}^k (\widetilde{C_M }d_{cc} ((x,u_k,\ldots , u_0) , (y,v_k,\ldots , v_0)))^{1/(k+1-s)}\\
&\le C(k+1) \widetilde{C_M} \diam([-M,M]^{k+2})\cdot d_{cc} ((x,u_k,\ldots , u_0) , (y,v_k,\ldots , v_0)))^{1/(k+1)}.
\end{align*}
For the last inequality, we used that 
\begin{align*}
&	d_{cc} ((x,u_k,\ldots , u_0) , (y,v_k,\ldots , v_0)))^{1/(k+1-s)} \\
&\qquad\le \diam([-M,M]^{k+2})\cdot d_{cc} ((x,u_k,\ldots , u_0) , (y,v_k,\ldots , v_0)))^{1/(k+1)}
\end{align*}
for all $s=0,\ldots , k$. 
This proves that $V_{f,t}$ is locally $\frac{1}{k+1}$-H\"older as desired.
\end{proof}

As a consequence, $V_{f,t}$ cannot increase Hausdorff dimension by more than a factor of $k+1$.
\begin{corollary}\label{Vft-cor}
	For all $t\in \R$, $f\in C^\infty(\R)$, and Borel sets $E\subset J^k(\R)$, 
\[
	\dim_{\Hau} (V_{f,t}(E)) \le \min \left\{(k+1)\dim_{\Hau}(E), \frac{(k+1)(k+2)}{2}\right\}.
\]
\end{corollary}


\section{Possible pairs for Hausdorff dimensions of images}

For all $f\in C^\infty (\R)$, $t\in \R$, and Borel sets $E$,
\[
	0 \le \dim_{\Hau}(J_{f,t}(E)) \le \dim_{\Hau}(\im(j^k(f))) = 1
\]
and
\[
	0 \le \dim_{\Hau}(V_{f,t}(E)) \le \dim_{\Hau}(\{x=t\}) =  \frac{(k+1)(k+2)}{2}.
\]
In this section, we will prove that for all $f\in C^\infty(\R)$, $t\in \R$ and pairs $(\alpha,\beta) \in [0,1]\times \left[ 0 , \frac{(k+1)(k+2)}{2}\right],$ there exists a Borel set $E\subset \R$ such that 
\[
	\dim_{\Hau}(J_{f,t}(E)) = \alpha \quad\text{and}\quad \dim_{\Hau}(V_{f,t}(E)) = \beta.
\]
(Theorem \ref{satisfying-theorem}). 
We will first prove the result for $\beta =0$ (Proposition \ref{large-Jft}) and then prove it for $\alpha=0$ (Corollary \ref{vertical-dimension-theorem}).
The desired set will be given by their union.


\subsection{Sets that are $J_{f,t}$-large and $V_{f,t}$-null}\label{sec:sets-large-proj}

In this section, $t\in \R$ and $f\in C^\infty(\R)$ will be fixed throughout.
Next section, we will construct sets  independent of $f$ that are large, null after being mapped by $V_{f,t}$, $J_{f,t}$, respectively.  

As $j^k(f):\R\to J^k(\R)$ is locally biLipschitz, it preserves Hausdorff dimension.
Hence, for all Borel sets $E\subset J^k(\R)$,
\[
	0 \le \dim_{\Hau}(J_{f,t}(E)) \le \dim_{\Hau} (\im (j^k(f))) = \dim_{\Hau} (\R) = 1. 
\]
Moreover, by the argument in Example \ref{jet-example-sec3}, if $E_\alpha \in \R$ has Hausdorff dimension $\alpha\in [0,1]$ (see Theorem \ref{Cantor-exist} below), then
\[
	 \dim_{\Hau}(J_{f,t}(j^k(f)(E_\alpha))) = \dim_{\Hau} (j^k(f)(E_\alpha)) = \dim_{\Hau}(E_\alpha) =\alpha.
\]
This shows that the full range of values for $\dim_{\Hau}(J_{f,t}(E))$ is possible as $E$ varies.
We can actually  prove an even stronger statement.

We first recall a well-known result about the existence of Cantor-like sets of every dimension in $\R$.
\begin{theorem}\label{Cantor-exist}(see for instance \cite[Section 4.10]{M:GOS})
For all $\alpha \in [0,1]$, there exists a compact set $E_\alpha \subset \R$ with $\dim_{\Hau}(E_\alpha) =\alpha$. 
\end{theorem}

We can now prove our main result of the section.
\begin{proposition}\label{large-Jft}
Fix $t\in \R$, $f\in C^\infty(\R)$, and $0\le \alpha\le 1$. 
Let $E_\alpha\subset\R$ be a compact set with $\dim_{\Hau}(E_\alpha) = \alpha$.
For all $p \in \{x=t\}$, the compact set 
\[
p\odot j^k(f)(E_\alpha):= \{p\odot j_x^k(f) :x\in E_\alpha\}\subset J^k(\R)
\]
 satisfies
\[
	 \dim_{\Hau}J_{f,t}(p\odot j^k(f)(E_\alpha) ) =\alpha
\qquad\text{and}\qquad 
\dim_{\Hau} V_{f,t}(p\odot j^k(f)(E_\alpha)  )= 0. 
\]
\end{proposition}

\begin{proof}
The proof is quite simple compared to what we have seen thus far.

Let $t,$ $ f, $ $E_\alpha,$ and $p$  be as in the statement. 
For all $x\in E_\alpha$, 
\[
	J_{f,t} (p\odot j_x^k(f)) = j_x^k(f),
\]
which implies
\[
	J_{f,t}(p\odot j^k(f)(E_\alpha) ) = \{j_x^k(f): x\in E_\alpha\}.
\]
By Proposition \ref{RW-obs}, 
\[
	\dim_{\Hau} J_{f,t}(p\odot j^k(f)(E_\alpha) )  = \dim_{\Hau} j^k(f)(E_\alpha) = \dim_{\Hau}(E_\alpha) =\alpha. 
\]

On the other hand, 
\[
	V_{f,t}(p\odot j_x^k(f)) = p
\]
for all $x\in E_\alpha$. 
This implies 
\[
	V_{f,t} (p\odot j^k(f)(E_\alpha) )  = \{p\}, 
\]
hence $\dim_{\Hau}V_{f,t} (p\odot j^k(f)(E_\alpha) )  = 0 .$
\end{proof}

This proposition begs a couple of questions.
\begin{question}
Fix $t\in \R$ and $\alpha \in (0,1]$. Does there exist a set $F_\alpha \subset J^k(\R) $ such that  
\[
	 \dim_{\Hau}J_{f,t}(F_\alpha ) =\alpha
\qquad\text{and}\qquad 
\dim_{\Hau} V_{f,t}(F_\alpha  )= 0
\]
for all $f\in C^\infty(\R)$?
\end{question} 

\begin{question}
Fix $f\in  C^\infty(\R)$ and $\alpha \in (0,1]$. Does there exist a set $G_\alpha \subset J^k(\R) $ such that  
\[
	 \dim_{\Hau}J_{f,t}(G_\alpha ) =\alpha
\qquad\text{and}\qquad 
\dim_{\Hau} V_{f,t}(G_\alpha  )= 0
\]
for all $t\in \R$?
\end{question}


\subsection{Sets that are $J_{f,t}$-null and $V_{f,t}$-large}\label{sec:sets-null-proj}

We now consider the complementary problem to the one considered in the previous section.
Fix $t\in \R$ and $f\in C^\infty(\R)$. 
We will show in this section that
$
	\dim_{\Hau} (\{x=t\})  = \frac{(k+1)(k+2)}{2}. 
$
Assuming this for now, 
\[
	0\le \dim_{\Hau}(V_{f,t} (E)) \le \dim_{\Hau} (\{x=t\} )= \frac{(k+1)(k+2)}{2}
\]
for all $E\subset J^k(\R)$. 
We will show that the full range of values for $\dim_{\Hau}(V_{f,t}(E))$ is possible.
In fact, the sets $E$ will be null sets when mapped by $J_{f,t}$  to $\im (j^k(f))$. 
Moreover, the sets $E$ we construct in this section will be independent of $f$, unlike what was constructed in the previous section. 
We will prove the following two results.

\begin{theorem}\label{vertical-dimension-theorem}
	Fix $t\in \R$ and $0\le \beta \le \frac{(k+1)(k+2)}{2}$. 
There exists a compact set $E_{t,\beta} \subset \{x=t\}$ such that 
\[
	\dim_{\Hau} E_{t,\beta} = \beta.
\]
\end{theorem}

\begin{corollary}\label{large-Vft-cor}
	Fix $t\in \R$ and $0\le \beta \le \frac{(k+1)(k+2)}{2}$. 
There exists a compact set $E_{t,\beta} \subset \{x=t\}$ such that for all $f\in C^\infty(\R)$,
\[
	\dim_{\Hau} V_{f,t}(E_{t,\beta})  = \beta \quad\text{and}\quad \dim_{\Hau} J_{f,t}(E_{t,\beta}) = 0. 
\]
\end{corollary}

Note that the corollary follows easily from the theorem. 
Indeed, for all $f\in C^\infty(\R)$, the restriction  $V_{f,t}|_{\{x=t\}}  : \{x=t\} \to \{x=t\}$ is biLipschitz (Proposition \ref{bilip-Vft}) and $J_{f,t}(\{x=t\}) =\{ j_0^k(f)\}$.

Given a set $E\subset \R$, our intuition based on the form of the dilations on $J^k(\R)$ says
\begin{equation}\label{sum-dimensions}
	\dim_{\Hau} (\{t\}\times [0,1]^{j-1}\times E \times \{0\}^{k+1-j} ) = \frac{j(j-1)}{2} + j \dim_{\Hau}  (E), \qquad j \le k+1.
\end{equation}
Then it's simple algebra to figure out what we should take $j$ and $E$ to be to construct the set $E_{t,\beta}$ in Theorem \ref{vertical-dimension-theorem}.
The main purpose of this section is to validate our intution and prove (\ref{sum-dimensions}) holds true.  

To prove Theorem \ref{vertical-dimension-theorem}, we first state a variant of the  Mass Distribution Principle. 
This result is well-known, so we will not prove it here and refer the reader to \cite[Lemma 1.2.8]{CP:FIP}.
\begin{lemma}\label{MDP}{(Mass Distribution Principle)}
Fix $s>0$ and a metric space $A$.
Suppose there exists a Borel measure $\mu$ on $A$ and constants $\delta,C>0$ such that 
\[
	\mu(B(x,r)) \le Cr^s
\]
for all $x \in A$ and $r<\delta$.
Then $\mathcal{H}^s(A)>0$. 
\end{lemma}

Howroyd proved that the converse is also true, showing that Frostman's Lemma holds for compact metric spaces \cite{H:OD}. 
We recommend the reader look at Mattila's treatment of the topic \cite[Section 8]{M:GOS}. 

\begin{theorem}\label{Frostman}\cite{H:OD}
Suppose  $A$ is a compact metric space with $\mathcal{H}^s(A)>0$. 
There exists $\delta>0$ and a Radon measure $\mu$ on $A$ satisfying $\mu(A)>0$ and
\[
	\mu(E) \le \diam(E)^s \quad \text{for all } E\subset A \text{ with } \diam(E) <\delta. 
\]
\end{theorem}

We now have a lower bound on the Hausdorff dimension of products of sets.
As the proof is identical as the Euclidean case (now using Lemma \ref{MDP} and Theorem \ref{Frostman}), we will simply refer the reader to the proof of Theorem 8.10 in \cite{M:GOS}. 

\begin{corollary}\label{dimension-cor}
Let $(A,d_A)$, $(B,d_B) $ be compact metric spaces.
Equip $A\times B$ with the metric 
\[
	d_{A\times B}((a_1,b_1),(a_2,b_2)) := d_A (a_1,a_2) + d_B(b_1,b_2).
\]
Then
\[
	\dim_{\Hau}(A) + \dim_{\Hau}(B) \le \dim_{\Hau} (A\times B).
\]
\end{corollary}

The opposite inequality does not hold in general. For example, Hatano constructed compact sets $E_1, E_2\subset \R^2$ of Hausdorff dimension $0$ for which $\dim_{\Hau} (E_1\times E_2)  = 1$ \cite[Theorem 4]{H:NOH}. 
Fortunately, we will not need the statement in full generality.

\begin{lemma}
Let $(A,d_A)$ be a compact metric space.
For $1\le\alpha<\infty$, let $[0,1]_\alpha$ be the interval $[0,1]$ equipped with the metric $d_\alpha(a_1,a_2) := |a_1-a_2|^{1/\alpha}$.
Then
\[
	\dim_{\Hau}([0,1]_\alpha\times A) = \alpha + \dim_{\Hau} (A). 
\]
\end{lemma}

\begin{proof}
By Corollary \ref{dimension-cor}, 
\[
		\dim_{\Hau}([0,1]_\alpha\times A) \ge \alpha + \dim_{\Hau} (A). 
\]
It remains to show the reverse inequality, which is just a standard exercise in measure theory.
We will include the proof here for completeness. 

Fix $\epsilon ,\delta>0$.
As $\mathcal{H}_\delta^{\dim_{\Hau}(A) +\epsilon}(A) =0$,
 there exists a covering $\{E_i\}$ of $A$ satisfying $r_i:=\diam(E_i)<\delta$ for all $i$ and 
\[
	\sum_i r_i^{\dim_{\Hau}(A)+\epsilon}<\delta.
\]
We may assume $r_i> 0$ for all $i$. 

For each $i$, cover $[0,1]$ with approximately $\frac{1}{r_i^\alpha}$ intervals $I_{ij}$ of length $r_i^\alpha$ (so each of these intervals will  have diameter $r_i$ in $[0,1]_\alpha$).
Then $\{I_{ij} \times E_i\}_{i,j}$ is a covering of $[0,1]_\alpha\times A$ with
\[
	\diam (I_{ij} \times E_i) = 2r_i <2\delta. 
\]	
Moreover,
\begin{align*}
	\sum_{i,j} \diam (I_{ij}\times E_i)^{\alpha + \text{dim}_{\Hau}(A)+\epsilon} 
& = \sum_{i,j} (2r_i)^{\alpha +\dim_{\Hau}(A) +\epsilon} \\
&\approx \sum_i \frac{1}{r_i^\alpha} \cdot (2r_i)^{\alpha +\dim_{\Hau}(A) +\epsilon} \\
& < 2^{\alpha+\dim_{\Hau} (A) +\epsilon }\delta,
\end{align*}
which shows
\[
	\mathcal{H}_{2\delta}^{\alpha+\dim_{\Hau}(A)+\epsilon} ([0,1]_\alpha \times A) \lesssim 2^{\alpha+\dim_{\Hau}(A)+\epsilon} \delta. 
\]	
As $\delta,\epsilon>0$ are arbitary, we may conclude 
\[
	\dim_{\Hau} ([0,1]_\alpha\times A) = \alpha +\dim_{\Hau}(A). 
\]
\end{proof}

As a consequence, we can use induction to calculate the Hausdorff dimension of a product of a box with a compact set.

\begin{proposition}\label{snowflake-dimension}
For $j\in \N_{\ge 2}$, equip $\R^j$ with the metric
\[
	d_j ((x_1,\ldots , x_j) , (y_1,\ldots ,y_j)) = \sum_{i=1}^j |x_i -y_i|^{1/i}. 
\]
For all compact sets $E\subset \R$,
\begin{align*}
	\dim_{\Hau} ([0,1]^{j-1}\times E)& = 1+ \cdots + (j-1) + j\dim_{\Hau}(E)\\
& = \frac{j(j-1)}{2} + j\dim_{\Hau}(E),
\end{align*}
where $[0,1]^j\times E$ is equipped with the restriction of  $d_j$. 
\end{proposition}

We can now prove our main result of the section, Theorem \ref{vertical-dimension-theorem}. 

\begin{proof}[Proof of Theorem \ref{vertical-dimension-theorem}]
Fix $t\in \R$ and $0\le \beta\le \frac{k(k+1)}{2}.$
If $\beta=0$, define $E_\beta=\emptyset$.
Hence assume otherwise.

For each nonnegative integer $j$, define $h_j = \frac{j(j+1)}{2}$. 
Choose the unique index $j$ such that $h_{j-1} < \beta\le h_j$.
As $0<\frac{\beta-h_{j-1}}{j}\le 1$, there exists a compact set $\widetilde{E_\beta}\subset \R$ with $\dim_{\Hau} (\widetilde{E_\beta}) = \frac{\beta - h_{j-1}}{j}$ (the existence of which is guaranteed by Theorem \ref{Cantor-exist}). 
By Proposition \ref{snowflake-dimension}, $[0,1]^{j-1}\times \widetilde{E}_\beta \subset (\R^j,d_j)$ has Hausdorff dimension equal to 
\[
	\frac{(j-1)j}{2} + j \left( \frac{\beta-h_{j-1}}{j}\right) = \beta.
\]
As shown in Section \ref{sec:havm}, the gauge distance takes the form
\[
	d((t,u_k,\ldots , u_0), (t,v_k,\ldots , v_0)) =\sum_{j=0}^k |x_j-y_j|^{1/(k+1-j)}
\]
on $\{x=t\}$.
If we define $E_{t,\beta}:= \{t\}\times [0,1]^{j-1} \times \widetilde{E}_\beta\times \{0\}^{k-j+1}\subset \{x=t\}$, then
\[
	\dim_{\Hau} (E_{t,\beta}) = \beta. 
\]
\end{proof}

Now fix $f\in C^\infty(\R)$ and $t\in \R$.
Also let $0\le \alpha\le 1$ and $0\le \beta \le \frac{(k+1)(k+2)}{2}$ be given. 
By Proposition \ref{large-Jft}, there exists $F_\alpha\subset J^k(\R)$ satisfying
\[
	\dim_{\Hau} J_{f,t}( F_\alpha) = \alpha \quad\text{and}\quad \dim_{\Hau} V_{f,t}( F_\alpha)=0.
\]
By Corollary \ref{large-Vft-cor}, there exists  $E_{t,\beta}\subset J^k(\R)$ for which 
\[
	\dim_{\Hau}V_{f,t} (E_{t,\beta})  = \beta \quad\text{and}\quad \dim_{\Hau} J_{f,t}(E_{t,\beta}) = 0. 
\]
Observe then that
\[
	\dim_{\Hau} J_{f,t} (F_\alpha \cup E_{t,\beta})  =  \alpha \quad\text{and}\quad \dim_{\Hau}V_{f,t} (F_\alpha \cup E_{t,\beta})  = \beta .
\]
We have proven the following satisfying result.
\begin{theorem}\label{satisfying-theorem}
Fix $f\in C^\infty(\R)$ and $t\in \R$.
For all $E\subset J^k(\R)$, 
\[
	0\le \dim_{\Hau} J_{f,t}(E) \le 1 \quad\text{and}\quad 0\le \dim_{\Hau} V_{f,t}(E) \le\frac{(k+1)(k+2)}{2}.
\]
Moreover, for all pairs $(\alpha,\beta) \in [0,1]\times \left[ 0,\frac{(k+1)(k+2)}{2}\right]$,  there exists $E_{\alpha,\beta}\subset J^k(\R)$ (depending on $f$ and $t$) satisfying
\[
		\dim_{\Hau} J_{f,t} (E_{\alpha,\beta})  =  \alpha \quad\text{and}\quad \dim_{\Hau}V_{f,t} (E_{\alpha,\beta})  = \beta .
\]
\end{theorem}


\section{Effects of $J_{f,t}$ and $V_{f,t}$ on Hausdorff dimension}\label{sec:effects-proj}

In Section \ref{sec:regularity}, we briefly touched on the topic of how $J_{f,t}$ and $V_{f,t}$ affect the Hausdorff dimension of sets. 
We showed that for all $t\in \R$, $f\in C^\infty(\R)$, and $E\subset J^k(\R)$,
\[
	\dim_{\Hau}(J_{f,t}(E)) \le \min \{\dim_{\Hau}(E),1\}
\]
and
\[
	\dim_{\Hau} (V_{f,t}(E)) \le \min \left\{(k+1)\dim_{\Hau}(E), \frac{(k+1)(k+2)}{2}\right\}
\]
(Corollaries \ref{Jft-cor} and \ref{Vft-cor}, respectively).

In the previous two sections, we  considered the dimensions of images of sets by $J_{f,t}$ and by $V_{f,t}$, but we didn't emphasize the dimensions of the sets themselves. 
We will do that here.
We first restate Theorem \ref{satisfying-theorem}, showing that  nearly all possibilities for the pairs $\text{dim}_{\Hau} (E), \dim_{\Hau}(J_{f,t}(E))$ are possible after taking into account Corollary \ref{Jft-cor}.
Note $\dim_{\Hau}J^k(\R) = 1 + \frac{(k+1)(k+2)}{2}$. 

\begin{proposition}\label{full-range-dim-Jft}
Fix $t\in \R$ and $f\in C^\infty(\R)$. 
For all $0\le \alpha\le 1$ and $\alpha \le \mu \le \frac{(k+1)(k+2)}{2}$, there exists a set $E= E_{\alpha,\mu,f,t}\subset J^k(\R)$ with $\dim_{\Hau} J_{f,t} (E) = \alpha$  and $\dim_{\Hau}E=\mu$.
\end{proposition}

\begin{proof}
The set $E_{\alpha,\mu}$ from Theorem \ref{satisfying-theorem} works. 
\end{proof}

The question of whether the full range of possibilities for pairs $(\text{dim}_{\Hau} E, \dim_{\Hau}V_{f,t}(E))$ is  much more difficult, as Hausdorff dimension may increase after mapping by $V_{f,t}$.
To illustrate this, we will  consider the problem when $f$ is a nonconstant linear function.

\begin{example}\label{line-example}
Suppose $f(x) = mx+b$ for $m\ne 0$. 
By the calculation performed in the proof of Proposition \ref{Vft-Holder},
\[
	(V_{f,t}(j_x^k(f))^{-1}\odot V_{f,t}(j_y^k(f)))_s = 0
\]
for $x,y \in \R, \ s=1,\ldots , k$.
We also have 
\[
	(V_{f,t}(j_x^k(f))^{-1}\odot V_{f,t}(j_y^k(f)))_0 = (t-y)m - (t-x)m = m(x-y).
\]
This implies
\[
	d(V_{f,t}(j_x^k(f)), V_{f,t}(j_y^k(f))) = \sqrt[k+1]{|m(x-y)|}. 
\]
By Propositions \ref{Gauge-cc} and \ref{RW-obs}, for each compact set $K\subset \R$, there is a constant $C =C(K)>0$ such that 
\[
	d_{cc}(j_x^k(f), j_y^k(f))^{1/(k+1)} \approx d_{cc}(V_{f,t}(j_x^k(f)), V_{f,t}(j_y^k(f)))
\]
for all $x,y\in K$. 
 This implies for all Borel sets $\tilde{E}\subset \R$ and $t\in \R$, 
\[
	\dim_{\Hau} V_{f,t}(j^k(f)(\tilde{E})) =(k+1)\dim_{\Hau}(j^k(f)(\tilde{E})) =  (k+1) \dim_{\Hau}(\tilde{E}) .
\]
\end{example}

We have shown the following:
\begin{proposition}\label{linear-k1}
For all linear functions $f(x) = mx+b$, $m\ne 0$, and $0 \le \mu\le 1$, there exists a Borel set $E\subset J^k(\R)$ such that for all $t\in \R$, 
\[
	\dim_{\Hau} (E) = \mu\quad\text{and}\quad \dim_{\Hau} (V_{f,t}(E)) = (k+1)\mu. 
\]
\end{proposition}

By appending a set to $E$ in the previous proposition, we can show that even more pairs $(\dim_{\Hau}(E), \dim_{\Hau}(V_{f,t}(E)))$ are possible (where now $t$ is fixed).  

\begin{theorem}
For all $t\in \R$, linear functions $f(x)= mx+b$, $m\ne 0$, $0\le \mu\le 1$, and $0 \le \beta \le (k+1)\mu$, there exists a Borel set $E_{t,f,\mu,\beta}\subset J^k(\R)$ such that 
\[
	\dim_{\Hau} (E_{t,f,\mu,\beta}) = \mu\quad\text{and}\quad \dim_{\Hau} (V_{f,t}(E_{t,f,\mu,\beta})) = \beta.
\]
\end{theorem}

\begin{proof}
Fix $t,f,\mu, \beta$ as in the statement. 
By Proposition \ref{linear-k1}, there exists a set $E\subset J^k(\R)$ such that
\[
	\dim_{\Hau}(E) = \frac{\beta}{k+1} \quad\text{and}\quad \dim_{\Hau}(V_{f,t}(E)) = \beta. 
\]
For a set $\tilde{F}\subset \R$ with $\dim_{\Hau}(\tilde{F}) = \mu$, define
\[
	F := (t,0,\ldots   ,0) \odot j^k(f)(\tilde{F})  = \{(t,0,\ldots , 0) \odot j^k_x(f):x\in \tilde{F}\}. 
\]
We showed in Proposition \ref{large-Jft} that
\[
	\dim_{\Hau}(V_{f,t}(F)) = 0. 
\]
Moreover, by left-invariance and Proposition \ref{RW-obs},
\[
	\dim_{\Hau}(F) = \dim_{\Hau} j^k(f)(\tilde{F}) = \dim_{\Hau}(\tilde{F})  = \mu.
\]
This implies
\[
	\dim_{\Hau}(E\cup F) = \max\{ \mu, \beta/(k+1)\} = \mu \quad\text{and}\quad \dim_{\Hau}(V_{f,t}(E\cup F) ) = \beta,
\]
and $E_{t,f,\mu,\beta} := E\cup F$ works. 
\end{proof}

It's not clear if one could prove a similar result if $\mu$ is allowed to vary between $0$ and $\frac{k+2}{2}$. 
It also isn't clear what happens when one considers more general functions or even constant functions. 
The nonconstant linear function case is  simple because many terms drop out when calculating the coordinates of $V_{f,t}(j_x^k(f))^{-1}\odot V_{f,t}(j_y^k(f))$ for $x,y\in \R$. 
We will conclude this section by stating the problem in general. 

\begin{question}
Fix $t\in \R$ and $f\in C^\infty(\R)$.
For which $0<\mu \le \frac{k+2}{2}$ and $0 < \beta \le (k+1)\mu$ does there exist a set $F = F_{\mu,\beta,t,f}$ satisfying $\dim_{\Hau} F = \mu$ and $\dim_{\Hau}V_{f,t}(F) = \beta$?
What if $f$ is assumed to be a constant linear function?
\end{question}

\section{Mappings and topological dimension}\label{sec:top-dim}

In Example \ref{plane-example}, we broached the question of how the mappings $J_{f,t}$ and $V_{f,t}$ affect topological dimension.
We will continue that study in this section by proving that all of the possible pairs $(\dim_{\Top}J_{f,t}(E), \dim_{\Top}V_{f,t}(E))$ can be attained as $E$ varies over Borel sets while $t$ and $f$ are fixed. 
In fact, the set $E$ can be chosen independently of $f$ (if we ask for $\dim_{\Top}V_{f,t}(E)\ge 1$). 

In Example \ref{plane-example}, we saw that for all $t\in \R$ and $p\in \{x=t\}$,  
\[
	V_{f,t} (p) = p - j_0^k(f).
\]
In particular, for all $E\subset \{x=t\},$
\[
	\dim_{\Top}V_{f,t}(E) = \dim_{\Top}(E) \quad\text{and}\quad \dim_{\Top}J_{f,t}(E) =0
\]
for all $f\in C^\infty(\R)$. 
Moreover, for all $p\in \{x=t\}$,
\[
	\dim_{\Top} V_{f,t}(p\odot \im (j^k(f))) = 0 \quad\text{and}\quad \dim_{\Top} J_{f,t}(p\odot \im (j^k(f))) =1
\]
(compare with set constructed in proof of Proposition \ref{large-Jft}).

By taking the union of these two sets, we have the following:
\begin{proposition}
Fix $f\in C^\infty(\R)$ and $t\in \R$.
For all pairs $(a,b) \in \{0,1\}\times \{0,\ldots , k+1\}$, there exists a set $E_{a,b}$ satisfying
\[
		\dim_{\Top}J_{f,t}(E_{a,b}) = a \quad\text{and}\quad \dim_{\Top}V_{f,t}(E_{a,b}) =b.
\]
\end{proposition}
This set is highly dependent on $f$ and $t$. 
Hence, we could ask if the set could be constructed to have the desired dimensions independent of $f$ or $t$. 
It turns out that we can construct the set to be independent of $t$ (at least for $\dim_{\Top}V_{f,t}(E)\ge 1$).

\begin{proposition}\label{top-ind-f}
For all $t\in \R$ and pairs $(a,b) \in  \{0,1\}\times \{1,\ldots , k+1\}$, there exists a set $F_{a,b} $ satisfying
\[
	\dim_{\Top}J_{f,t}(F_{a,b}) = a \quad\text{and}\quad \dim_{\Top}V_{f,t}(F_{a,b}) =b
\]
for all $f\in C^\infty(\R)$.
Moreover, $\dim_{\Top}(F_{a,b}) = b. $
\end{proposition}

\begin{proof}
	Fix $t\in \R$ and a pair $(a,b)\in  \{0,1\}\times \{1,\ldots , k+1\}$.
Let $E\subset \{x=t\}$ be a set with $\dim_{\Top}(E) = b$.
Then $\dim_{\Top} V_{f,t}(E) = b$ and $\dim_{\Top}J_{f,t}(E) = 0$ (see Example \ref{plane-example}).  
If $a=0$, we are done.

Now suppose $a=1$. 
Consider $F:= \{(x,x,0,\ldots , 0):x\in (t+1,t+2)\}$.
We will show $\dim_{\Top}J_{f,t}(F) = 1$ and $\dim_{\Top}V_{f,t}(F) =1$ for all $f\in C^\infty(\R)$.
Then $F_{a,b}:= E\cup F$ will work to prove the proposition. 

To show every $J_{f,t}(F)$ and $V_{f,t}(F)$ have topological dimension one, it suffices to show each is the image of a nonconstant smooth curve.
Indeed, the image of every smooth curve in Euclidean space has Hausdorff dimension at most one, which bounds the topological dimension by one \cite[Theorem 8.14]{LOA:H}.  
On the other hand, the topological dimension must be at least one as a consequence of the Inverse Function Theorem. 

We see that $J_{f,t}(F)$ is a smooth nonconstant curve for all $f\in C^\infty(\R)$ since $J_{f,t}(x,x,0,\ldots , 0) = j_{x-t}^k(f)$ for all $x \in (t+1,t+2)$. 
To see $V_{f,t}(F)$ is a smooth curve, note that
\[
	V_{f,t}(x,x,0,\ldots  0 ) = (x,x,0,\ldots  , 0) \odot j_{x-t}^k(f)^{-1}
\] 
for all $x\in (t+1,t+2)$.

Suppose, for contradiction, that $x\mapsto (x,x,0,\ldots  , 0) \odot j_{x-t}^k(f)^{-1}$, $x\in (t+1,t+2)$, is constant for some $f\in C^\infty(\R)$.
Then there is $c\in \R$ such that 
\[
	c = ((x,x,0,\ldots  , 0) \odot j_{x-t}^k(f)^{-1})_k= x- f^{(k)}(x-t)
\]
for all $x\in (t+1,t+2)$. 
This implies for some $d\in \R$, \[f^{(k-1)}(x-t) = \frac{1}{2}x^2 -cx + d\] for all $x\in (t+1,t+2)$.
Then by Proposition \ref{right-multiplication}, 
\begin{align*}
	((x,x,0,\ldots  , 0) \odot j_{x-t}^k(f)^{-1})_{k-1} 
&= -f^{(k-1)}(x-t) -(x-t)(x-  f^{(k)}(x-t))\\
& = - \frac{1}{2}x^2 +cx - d -c(x-t)
\end{align*}
must also be constant, but it is clearly not by the presence of the $-\frac{1}{2}x^2$ term.
This proves that $V_{f,t} $ is a nonconstant smooth curve, and the proposition follows.
\end{proof}

The question of whether we can construct the set independently seems much more difficult, and we will leave it open. 

\begin{question}
 For a fixed $f\in C^\infty(\R)$ and pair $(a,b) \in \{0,1\}\times \{0,\ldots , \frac{(k+1)(k+2)}{2}\}$, does there exist a set $G_{a,b}$ satisfying 
\[
\dim_{\Top}J_{f,t}(G_{a,b}) = a \quad\text{and}\quad \dim_{\Top}V_{f,t}(G_{a,b}) =b
\]
for all $t\in \R$?
\end{question} 


\section{Open questions and mappings of general jet spaces}\label{sec:open-proj}

We conclude with a few open questions related to fixing a Borel set and letting $t$ or $f$ vary. 
We also note that an analogous construction could be performed to define mappings of general jet space Carnot groups. 

We have not considered how $\dim_{\Hau}V_{f,t}(E)$ changes as $t$ varies while $f$ and $E$ are fixed. 
For fixed $t\in \R$, $p\in \{x=t\}$, and $f\in C^\infty(\R)$, we showed 
\[
	\dim_{\Hau} V_{f,t}(p\odot \im j^k(f))= 0 .
\]
For all $p\in \{x=t\}$ and $s\ne t$, the map $x \mapsto V_{f,s}(p\odot j^k_x(f))$ will be a nonconstant curve if $f$ isn't a constant function.
This implies
\[
	\dim_{\Hau}V_{f,s}(p\odot\im j^k(f)) \ge 1. 
\]
This motivates the following general question.
\begin{question}
	Fix $ f\in C^\infty(\R)$ and a Borel set $E\subset J^k(\R)$.
What is the behavior of the function $t\mapsto \dim_{\Hau}V_{f,t}(E)$? 
For example, does it attain a finite number of values? 
What is its regularity? 
Is there a value that it attains almost everywhere?
\end{question}

The final set of questions we pose seems to be the most difficult one. 
What happens if we let the function $f$ vary as $t$ and $E$ are fixed?
Recall the notion of prevalence in the sense of Hunt, Sauer, and Yorke \cite{HSY:PA, OY:P, SY:ATD}.
Prevalence provides a notion of ``almost everywhere'' in an infinite-dimensional Banach space, such as $C^{k+1}(\R)$.
In 2013, Balogh, Tyson, and Wildrick showed that the set of Newtonian-Sobolev functions that maximally increase Hausdorff dimension is prevalent within a certain Newtonian-Sobolev space \cite[Theorem 1.2]{BTW:DD}. 
One could study an analogous problem in our setting.

\begin{question}
Fix $t\in \R$ and a Borel set $E \subset J^k(\R)$.
What are the behaviors of the functions $f\mapsto \dim_{\Hau} J_{f,t}(E)$ and $f\mapsto \dim_{\Hau}V_{f,t}(E)$, where $f$ ranges over $C^{k+1}(\R)$?
Are there topologies on $C^{k+1}(\R)$ under which these two functions are continuous?
Do there exist $\alpha,\beta\in \R$ for which the set of $f\in C^{k+1}(\R)$ satisfying $\dim_{\Hau} J_{f,t}(E) = \alpha$ and $\dim_{\Hau}V_{f,t}(E) =\beta$ is prevalent?
\end{question}

We conclude by noting that one could define  vertical and horizontal mappings in the same way in general jet space Carnot groups $J^k(\R^n)$ (see \cite[Section 4.4]{JSA:W} for  discussion and notation of these groups).
For the  vertical planes, one would take the codimension-$n$ planes $\{x=t\}:= \{(x,u^{(k)})\in J^k(\R^n):x=t\}$ for $t\in \R^n$. 
Then for every $f\in C^\infty(\R^n)$, $t\in \R^n$, and $p\in J^k(\R^n)$, there exist uniquely  $p_V \in \{x=t\}$ and $p_H\in \im (j^k(f))$ such that 
\[
	p = p_V\odot p_H. 
\]
One could then prove analogues of every result in this paper for general jet space Carnot groups.


\bibliographystyle{plain}
\bibliography{ThesisJung}

\end{document}